\documentclass[11 pt]{amsart}
\usepackage{amsmath, amsthm,amssymb,bbm,enumerate,mathrsfs,mathtools}
\usepackage[foot]{amsaddr}
\usepackage{a4wide}
\usepackage{tikz,xcolor,verbatim}
\usetikzlibrary{calc,shapes}
\usepackage{hyperref}
\usepackage{makecell}
\usepackage[numbers,sort&compress]{natbib}

\usepackage{lineno}
\nolinenumbers

\tikzset{
  bigblue/.style={circle, draw=blue!80,fill=blue!40,thick, inner sep=1.5pt, minimum size=5mm},
  bigred/.style={circle, draw=red!80,fill=red!40,thick, inner sep=1.5pt, minimum size=5mm},
  bigblack/.style={circle, draw=black!100,fill=black!40,thick, inner sep=1.5pt, minimum size=5mm},
  bluevertex/.style={circle, draw=blue!100,fill=blue!100,thick, inner sep=0pt, minimum size=2mm},
  redvertex/.style={circle, draw=red!100,fill=red!100,thick, inner sep=0pt, minimum size=2mm},
  blackvertex/.style={circle, draw=black!100,fill=black!100,thick, inner sep=0pt, minimum size=2mm},  
  whitevertex/.style={circle, draw=black!100,fill=white!100,thick, inner sep=0pt, minimum size=2mm},  
  grayvertex/.style={circle, draw=black!100,fill=black!50,thick, inner sep=0pt, minimum size=2mm}, 
  smallblack/.style={circle, draw=black!100,fill=black!100,thick, inner sep=0pt, minimum size=1mm},
  smallwhite/.style={circle, draw=black!100,fill=white!100,thick, inner sep=0pt, minimum size=1mm} 
}

\makeatletter
\pgfdeclareshape{myNode}{
  \inheritsavedanchors[from=rectangle] 
  \inheritanchorborder[from=rectangle]
  \inheritanchor[from=rectangle]{center}
  \inheritanchor[from=rectangle]{north}
  \inheritanchor[from=rectangle]{south}
  \inheritanchor[from=rectangle]{west}
  \inheritanchor[from=rectangle]{east}
  \backgroundpath{
    \southwest \pgf@xa=\pgf@x \pgf@ya=\pgf@y
    \northeast \pgf@xb=\pgf@x \pgf@yb=\pgf@y
    \pgfsetcornersarced{\pgfpoint{5pt}{5pt}}
    \pgfpathmoveto{\pgfpoint{\pgf@xa}{\pgf@ya}}
    \pgfpathlineto{\pgfpoint{\pgf@xa}{\pgf@yb}}
    \pgfpathlineto{\pgfpoint{\pgf@xb}{\pgf@yb}}
    \pgfsetcornersarced{\pgfpoint{5pt}{5pt}}
    \pgfpathlineto{\pgfpoint{\pgf@xb}{\pgf@ya}}
    \pgfpathclose
 }
}
\makeatother

\title[Characterizing Circular Colouring Mixing for $\frac{p}{q}<4$ ]{Characterizing Circular Colouring Mixing for $\frac{p}{q}<4$}

\author[Brewster]{Richard C. Brewster}\thanks{Research funded in part by the Natural Sciences and Engineering Council of Canada}
\address[Richard C. Brewster]{Department of Mathematics and Statistics, Thompson Rivers University, Kamloops, BC, Canada} 
\email{rbrewster@tru.ca}

\author[Moore]{Benjamin Moore}
\address[Benjamin Moore]{Institute of Computer Science, Charles University, Prague, Czech Republic}  
\email{brmoore@iuuk.mff.cuni.cz}

\date{}

\newtheorem{thm}[equation]{Theorem}
\newtheorem{lem}[equation]{Lemma}
\newtheorem{prop}[equation]{Proposition}
\newtheorem{conj}[equation]{Conjecture}
\newtheorem{cor}[equation]{Corollary}
\newtheorem{claim}[equation]{Claim}

\theoremstyle{definition}
\newtheorem{defn}[equation]{Definition}
\newtheorem{obs}[equation]{Observation}

\newtheoremstyle{case}{}{}{\normalfont}{}{\itshape}{\normalfont:}{ }{}

\theoremstyle{case}

\numberwithin{equation}{section}

\newcommand{\claimproof}{\renewcommand{\qedsymbol}{$\diamond$}}  

\newcommand\Hrec[1]{\textsc{\ensuremath{#1}-Recolouring}}
\newcommand\Hmix[1]{\textsc{\ensuremath{#1}-Mixing}}
\newcommand\Hret[1]{\textsc{\ensuremath{#1}-Retraction}}
\newcommand\Hfold[1]{\textsc{\ensuremath{#1}-Folding}}
\newcommand\Hcomp[1]{\textsc{\ensuremath{#1}-Compaction}}

\newcommand{\Col}{\operatorname{\mathbf{Col}}}

\begin{document}

\begin{abstract}
Given a graph $G$, the $k$-mixing problem asks: Starting with a $k$-colouring of $G$, can one obtain all $k$-colourings of $G$ by changing the colour of only one vertex at a time, while at each step maintaining a $k$-colouring? More generally, for a graph $H$, the $H$-mixing problem asks: Can one obtain all homomorphisms $G \to  H$, starting from one homomorphism $f$, by changing the image of only one vertex at a time, while at each step maintaining a homomorphism $G \to H$?

This paper focuses on a generalization of $k$-colourings, namely $(p,q)$-circular colourings. We show that when $2 < \frac{p}{q} < 4$, a graph $G$ is $(p,q)$-mixing if and only if for any $(p,q)$-colouring $f$ of $G$, and any cycle $C$ of $G$, the wind of the cycle under the colouring equals a particular value (which intuitively corresponds to having no wind).
As a consequence we show that $(p,q)$-mixing is closed under a restricted homomorphism called a fold. Using this, we deduce that $(2k+1,k)$-mixing is co-NP-complete for all $k \in \mathbb{N}$, and by similar ideas we show that if the circular chromatic number of a connected graph $G$ is $\frac{2k+1}{k}$, then $G$ folds to $C_{2k+1}$. We use the characterization to settle a conjecture of Brewster and Noel, specifically that the circular mixing number of bipartite graphs is $2$. 
Lastly, we give a polynomial time algorithm for $(p,q)$-mixing in planar graphs when $3 \leq \frac{p}{q} <4$.
\end{abstract}

\maketitle

\textbf{Keywords:} graph theory, colouring, mixing, homomorphism, reconfiguration

\section*{Acknowledgement}
We thank the referees for their helpful suggestions.
The second author is grateful for discussions with Vijay Subramanya and Naomi Nishimura.

\section{Introduction}

Combinatorial reconfiguration problems have received much attention in the recent literature~\cite{Ito, changeSurvey,naomisurvey}.  A particularly well-studied example is graph colouring reconfiguration~\cite{3colReconfig, Bonsma,marthe,itocolouring,dvok2020thomassentype}.  For a fixed positive integer $k$, the \Hrec{k} problem takes as input a graph $G$ together with two $k$-colourings $f$ and $g$ and asks if there is a sequence of $k$-colourings of $G$, $f=f_0, f_1, \dots, f_l = g$ such that successive colourings in the sequence differ on a single vertex. In the affirmative we say $f$ \emph{reconfigures} to $g$. Using a standard reconfiguration framework, we define $\Col(G,K_{k})$ to be the graph whose vertex set is the set of $k$-colourings of $G$ with two $k$-colourings $c_1, c_2$ adjacent if $c_1(v) \neq c_2(v)$ for at most one vertex $v \in V(G)$.  Clearly, $f$ reconfigures to $g$ precisely when there is a path from $f$ to $g$ in $\Col(G,K_{k})$. Cereceda, van den Heuvel, and Johnson~\cite{3colReconfig} show \Hrec{3} is polynomial time solvable.  This is somewhat surprising given the $3$-colouring problem itself is NP-complete. In comparison, Bonsma and Cereceda~\cite{Bonsma} show for all $k \geq 4$, the \Hrec{k} problem is PSPACE-complete. 

In this paper, we study a related reconfiguration problem.  For a fixed $k$, the \Hmix{k} problem takes as input a graph $G$ and asks if $\Col(G,K_{k})$ is connected. If yes, we say $G$ is \emph{$k$-mixing}. When $k \leq 2$, it is easy to see that the \Hmix{k} problem is in P, since $G$ is $2$-mixing if and only if $G$ has no edges. For $k =3$, Cereceda, van den Heuvel and Johnson~\cite{Mixing3Col} show the \Hmix{3} problem is co-NP-complete in general, and in P when the inputs are restricted to planar graphs.  The complexity of \Hmix{k} is open for all $k \geq 4$, and conjectured in~\cite{Mixing3Col} to be PSPACE-complete.

By studying $(p,q)$-circular colourings (a refinement of $k$-colourings defined in Section~\ref{sec:defn}) the dichotomy theorem in~\cite{3colReconfig} is extended in~\cite{circularReconfig} to the following.  The \Hrec{(p,q)} problem is polynomial if $2 \leq \frac{p}{q}< 4$ and is PSPACE-complete for $\frac{p}{q} \geq 4$. This result uses the ideas of~\cite{3colReconfig}, but the more general setting highlights the key tool in the polynomial time algorithm.  In particular, the notion of \emph{cycle wind} invariance (defined in Section~\ref{sec:defn}) follows from the restriction $\frac{p}{q} < 4$, rather than, say, the number of colours available.  In general, the role of topological ideas in homomorphism reconfiguration appear in Wrochna's ground breaking works~\cite{Wrochna, WROCHNA20181}.

In this work, we extend the study of the mixing problem to $(p,q)$-circular colourings. 
In brief, we prove the following characterization. For $2 < \frac{p}{q} < 4$, a graph $G$ is $(p,q)$-mixing if and only if for any $(p,q)$-circular colouring of $G$, and any cycle $C$ of $G$, the wind of the cycle under the colouring equals a particular value (which intuitively corresponds to having no wind).  This results extends the work of~\cite{Mixing3Col} from $k=3$ to the range $2 < \frac{p}{q} < 4$.

As a consequence, we show that $(p,q)$-mixing is closed under a restricted, wind preserving, homomorphism called a \emph{fold}. Using this, we prove that $(2k+1,k)$-mixing is co-NP-complete for all $k \in \mathbb{N}$, extending the $k=3$ result in~\cite{Mixing3Col}.  By similar ideas we show that if the circular chromatic number of a connected graph $G$ is $\frac{2k+1}{k}$, then $G$ folds to $C_{2k+1}$. This complements the result of Cook and Evans~\cite{Chromaticnumberandfolding} stating that a graph with chromatic number $k$ folds to $K_k$.  (Zhu~\cite{Zhusurvey} shows that such folding results for other values of $\frac{p}{q}$ is not possible.) We use the mixing characterization to settle a conjecture of Brewster and Noel in the positive, specifically that the circular mixing number of bipartite graphs is $2$. Lastly, we give a polynomial time algorithm for $(p,q)$-mixing for planar graphs when $3 \leq \frac{p}{q} <4$.

Thus, our work maybe viewed as a refinement of the $3$-mixing results in~\cite{Mixing3Col}.  By studying $(p,q)$-colourings for $2 \leq \frac{p}{q} < 4$ (of which $k = \frac{p}{q} = 3$ is a special case) we see some results extend to the entire range $2  \leq \frac{p}{q} < 4$, the co-NP-completeness results extend to odd cycles (of which $K_3$ is a special case) and finally the planar complexity results extend to $3 \leq \frac{p}{q} < 4$ (again of which $k=3$ is a special case).  In the latter two settings, we provide evidence that extending the folding results beyond odd cycles and the planar results to $2< \frac{p}{q} < 3$ will be challenging.

The following section contains key definitions, statements of our main results, and an outline of the paper.

\section{Definitions and statements of results}\label{sec:defn}

In this work we consider finite, simple graphs.  We refer the reader to~\cite{BondyMurty} for standard graph theoretic terminology and notation.

The key focus of this paper is extending results on $k$-mixing to circular colourings. We first define such colourings. Let $p, q$ be positive integers. A \emph{$(p,q)$-circular colouring} of a graph $G=(V,E)$ is a labelling of the vertices $f:V \to \{ 0, 1, 2, \dots, p-1 \}$ such that if $uv \in E$, then $q \leq |f(u)-f(v)| \leq p-q$.  That is, adjacent vertices receive labels differing by at least $q$ in the circular sense, i.e. $|x-y| = \min \{ |x-y|, |p-x+y| \}$.  For brevity we drop the word circular and use the term \emph{$(p,q)$-colouring} throughout the paper.  When $q=1$ the condition reduces to adjacent vertices receiving different colours, hence $(p,q)$-colourings generalize $k$-colourings.  Thus, one may view our work here and in~\cite{circularReconfig} as extending the works of Cereceda et al. from $k = 3$ to $2 < \frac{p}{q} < 4$.

It is convenient, and common, to define colourings via graph homomorphisms. (See~\cite{Graphsandhomomorphisms} for a detailed treatment of graph homomorphisms.) Given graphs $G$ and $H$, a \emph{homomorphism} $f$ from $G$ to $H$ is a mapping $f:V(G) \to V(H)$ such that for any edge $uv \in E(G)$, $f(u)f(v) \in E(H)$. We write $f: G \to H$ or simply $G \to H$ to indicate the existence of a homomorphism from $G$ to $H$.  Owing to the fact that a $k$-colouring of $G$ is a homomorphism $G \to K_k$, in general a homomorphism of $G \to H$ is called an \emph{$H$-colouring} of $G$.  Circular colourings are graph homomorphisms to specific target graphs called \emph{circular cliques}. (See~\cite{BondyandHell} for a complete development.) Let $p$ and $q$ be positive integers such that $\frac{p}{q} \geq 2$. Define the \textit{$(p,q)$-circular clique}, $G_{p,q}$, to have vertex set $V(G_{p,q}) = \{0,\ldots,p-1\}$ and edge set $uv \in E(G_{p,q})$ if and only if $ q \leq |u-v| \leq p-q$. Thus, a $(p,q)$-colouring of a graph $G$ is a homomorphism $G \to G_{p,q}$.  Note that $G_{k,1} \cong K_k$ and again we note circular colourings generalize $k$-colourings.

Bondy and Hell~\cite{BondyandHell} show $G_{p,q} \to G_{p',q'}$ if and only if $\frac{p}{q} \leq \frac{p'}{q'}$.  As homomorphisms compose, if $G \to G_{p,q}$, then $G \to G_{p',q'}$ for all $\frac{p'}{q'} \geq \frac{p}{q}$.    Moreover, in~\cite{BondyandHell}, they show there is a minimum $\frac{p}{q}$ such that $G \to G_{p,q}$.  We called this minimum value the \emph{circular chromatic number} and denote it by $\chi_c(G)$.  In particular, $\lceil \chi_c(G) \rceil = \chi(G)$ showing the circular chromatic number is a refinement of the chromatic number.  

One can ask reconfiguration questions about $H$-colourings for general $H$. There are many papers studying such questions, see for instance~\cite{reflexivedigraphs,Wrochna,WROCHNA20181,Frozen,lee2018reconfiguring}. Given graphs $G$ and $H$, the graph $\Col(G,H)$ has as vertices all homomorphisms from $G$ to $H$ with two homomorphisms $f$ and $g$ adjacent if $f(v) \neq g(v)$ for at most one vertex $v \in V(G)$. (This definition implies $\Col(G,H)$ is reflexive, i.e. has a loop on each vertex.  For readers familiar with the exponential graph $H^G$, the reconfiguration graph $\Col(G,H)$ is then a spanning subgraph of the reflexive subgraph of $H^G$. See for example~\cite{Graphsandhomomorphisms}.)  Thus, in general we define the \Hrec{H} problem as taking two homomorphism $f, g: G \to H$ and asking if there is a path from $f$ to $g$ in $\Col(G,H)$.  The \Hmix{H} problem takes as input a graph $G$ and asks if $\Col(G,H)$ is connected.

In this work we study the mixing problem for $(p,q)$-colourings.  Namely, for a graph $G$ we characterize when $\Col(G, G_{p,q})$ is connected and we examine the computational complexity of determining if $\Col(G,G_{p,q})$ is connected.  If $G$ is a yes instance to the \Hmix{G_{p,q}} problem, we say $G$ is \emph{$(p,q)$-mixing}. Of particular interest in this paper are $(2k+1,k)$-colourings.  Since $G_{2k+1,k} \cong C_{2k+1}$, our results on $(2k+1,k)$-mixing are often stated as results for odd cycles.
Our first two main results generalize the Cereceda, van den Heuvel, Johnson \Hmix{3} complexity results to odd cycles for general inputs and for planar inputs to all $3 \leq \frac{p}{q} < 4$. We note that Theorem~\ref{intromaintheorem} encompasses all known complexity results on the \Hmix{H} problem for general inputs.

\begin{thm}
\label{intromaintheorem}
The \Hmix{C_{2k+1}} problem is co-NP-complete for all $k \in \mathbb{Z}^{+}$.
\end{thm}

\begin{thm} 
\label{introplanarity}
Let $G$ be a planar graph and let $p$ and $q$ be positive integers where $3 \leq \frac{p}{q} < 4$. Then it can be determined in polynomial time whether $G$ is $(p,q)$-mixing. 
\end{thm} 

Based on Theorem~\ref{intromaintheorem}, the results in~ \cite{circularReconfig}, and the conjecture that the \Hmix{k} problem is PSPACE-complete for $k \geq 4$, the following conjecture is natural:

\begin{conj}
If $\frac{p}{q} = 2$, then \Hmix{G_{p,q}} is in P. If $2 < \frac{p}{q} < 4$, then \Hmix{G_{p,q}} is co-NP-complete. If $\frac{p}{q} \geq 4$, then \Hmix{G_{p,q}} is PSPACE-complete. 
\end{conj}

Note the \Hmix{G_{p,q}} problem is trivial when $\frac{p}{q} = 2$ as a graph $G$ is $(p,q)$-mixing if and only if $G$ has no edges.

Our third result and our key contribution is the following structural characterization of $(p,q)$-mixing when $2 < \frac{p}{q} < 4$ from which Theorem~\ref{intromaintheorem} and Theorem~\ref{introplanarity} follow.  To state the characterization we require the following definitions. Let $G$ be a graph and let $f$ be a $(p,q)$-colouring of $G$. Given an edge $uv$ of $G$, define the \emph{weight of $uv$ under $f$} as $W(uv,f) = (f(v) - f(u)) \bmod{p}$. Note there is an implied direction to $uv$ here.  The following two observations are used throughout the paper.

\begin{obs}\label{obs:negedge}
Let $G$ be a graph and $f$ be a $(p,q)$-colouring of $G$. 
\begin{list}{(\alph{enumi})}{\usecounter{enumi}}
  \item For an edge $uv$ in $G$, $W(uv,f) = - W(vu,f) \bmod{p} = p - W(vu,f)$.
  \item For a path of length 2, $uvw$, with $f(u) = f(w)$, $W(uv,f)+W(vw,f) = p$.
\end{list}
\end{obs}

Given a set of edges from $G$, $S \subseteq E(G)$, we naturally extend the concept of weight by
$$
W(S,f) = \sum_{uv \in S} W(uv,f).
$$
Given a walk $X = x_0, x_1, x_2, \dots, x_l$ in $G$, as a slight abuse of notation we write
$$
W(X,f) = \sum_{i=1}^l W(x_i x_{i-1}, f).
$$
In particular for a cycle of length $l$, $C = c_0, c_1, \dots, c_{l-1}, c_0$,
$$
W(C,f) = \sum_{i=1}^{l} W(c_i c_{i-1},f),
$$
where index arithmetic is modulo $l$.  For a cycle $C$, the sum telescopes and hence 
$$
W(C,f) = p \cdot w_{f}(C) \mbox{~for some integer~} w_f(C).
$$ 
We call $w_f(C)$ the \emph{wind} of $C$ under $f$.

As an example, consider two $(5,2)$-colourings of $C_{10}$ (on vertex set $\{ c_i~|~0 \leq i \leq 9\}$). Define $f$ by $f(c_i) = 2i \bmod{5}$ and $g(c_i) = 0$ if $i$ is even and $g(c_i)=2$ if $i$ is odd.  Then $W(C_{10},f) = 20$ and $w_f(C_{10}) = 4$. Similarly, $W(C_{10},g) = 25$ and $w_g(C_{10}) = 5$. On the other hand, any $(5,2)$-colouring of $C_8$ will have wind $4$.

Now we can state the theorem.

\begin{thm}
\label{introcharacterization}
Fix integers $p,q \in \mathbb{Z}$ such that $2 < \frac{p}{q} < 4$. A graph $G$ is $(p,q)$-mixing if and only if for every $(p,q)$-colouring $f$, and for every cycle $C$ in $G$, the wind of $C$ under $f$ is $\frac{|E(C)|}{2}$. 
\end{thm}

It follows from this characterization, that to certify $G$ is not $(p,q)$-mixing, we simply need to exhibit a $(p,q)$-colouring and a cycle without the required winding number. Hence, we have a concise NO-certificate and thus \Hmix{G_{p,q}} is in co-NP when $2 < \frac{p}{q} < 4$. 
The invariance of wind under reconfiguration turns out to be the key tool in both~\cite{Mixing3Col, circularReconfig} rather than say the number of colours available. Indeed the colourings under study in these previous works, this invariance can be shown to follow from the fact that four cycles must have wind 2. This is made explicit in~\cite{reflexivedigraphs}.

While Theorem~\ref{introcharacterization} characterizes $(p,q)$-mixing for all $2 < \frac{p}{q} < 4$, our co-NP-completeness result is for the special case of $(2k+1,k)$-colourings. (These are the $2$-regular circular cliques and thus extend the results for $3$-colourings in~\cite{Mixing3Col}). Our result for $(2k+1,k)$-mixing exploits a  restricted homomorphism called a fold, which we define now.  Given a graph $G$, and two vertices $x$ and $y$ at distance $2$, let $G_{xy}$ be the graph obtained by identifying $x$ and $y$. Call the new vertex $v_{xy}$.  The homomorphism $f: G \to G_{xy}$ defined by $f(x) = f(y) = v_{xy}$ and $f(u)=u$ for $u \not\in \{x,y\}$ is an \emph{elementary fold}. We say a graph $G$ \emph{folds} to a graph $H$ if there is a homomorphism $f: G \to H$ where $f$ is a composition of elementary folds and $f(G)$ is isomorphic to $H$. We call such a mapping a \emph{folding}.

We now state our next result.

\begin{thm}
\label{foldingcharacterizationintro}
Fix integers $p,q$ such that $2 < \frac{p}{q} < 4$. Suppose  $G$ folds to $H$. If $H$ is not $(p,q)$-mixing, then $G$ is not $(p,q)$-mixing. In particular, for $p = 2q+1$, a bipartite, connected graph $G$ is $(p,q)$-mixing if and only if $G$ does not fold to $C_{4q+2}$. 
\end{thm}

We use Theorem~\ref{foldingcharacterizationintro}  to resolve a conjecture in~\cite{BrewsterNoel} in the affirmative:

\begin{thm}
\label{RickJonconjecture}
For every bipartite graph $G$, there exists an integer $k \in \mathbb{Z}$ such that $G$ is $C_{2k+1}$-mixing.
\end{thm}

Theorem~\ref{intromaintheorem} follows from Theorem~\ref{foldingcharacterizationintro} by showing that testing if an input graph $G$ folds to $C_{4q+2}$ is NP-complete. For this final step, we build on a theorem of Vikas~\cite{vikas}. 
We conjecture that Theorem~\ref{foldingcharacterizationintro} can be generalized to all $2 < \frac{p}{q} < 4$ by allowing a finite family of forbidden folding targets.

\begin{conj}
\label{finitefamily}
Fix integers $p,q$ such that $2 < \frac{p}{q} < 4$. Then there is a finite family of graphs $\mathcal{G}_{p,q}$ such that $G$ is $(p,q)$-mixing if and only if $G$ does not fold to a graph $H \in \mathcal{G}_{p,q}$. 
\end{conj}

In Section~\ref{windingcharacterizationsection} we prove Theorem~\ref{introcharacterization}. Given two $(p,q)$-colourings $f$ and $g$ ($2 < \frac{p}{q} < 4$) of a graph $G$, the characterization in~\cite{circularReconfig} asserts $f$ reconfigures to $g$, if for every cycle $C$ in $G$, $w_f(C)=w_g(C)$, and further, the set of fixed vertices (vertices which can never change colour under reconfiguration) and their images are the same under both $f$ and $g$.  We show the existence of fixed vertices implies the existence of a cycle $C$ whose wind is not $\frac{|E(C)|}{2}$. We then argue that given a $(p,q)$-colouring $f$ and a cycle $C$ whose wind is not $\frac{|E(C)|}{2}$, we can define a new $(p,q)$-colouring $g$ such that $w_f(C) \neq w_g(C)$. The characterization in~\cite{circularReconfig} gives that $f$ does not reconfigure to $g$, and hence the graph is not $(p,q)$-mixing. The converse of Theorem~\ref{introcharacterization} follows similarly by appealing to the characterization in~\cite{circularReconfig}.

In Section~\ref{foldingsection} we prove Theorem~\ref{foldingcharacterizationintro}, Theorem~\ref{RickJonconjecture} and Theorem~\ref{intromaintheorem}. To prove Theorem~\ref{foldingcharacterizationintro} we analyse how folding impacts the winds of cycles, and then appeal to Theorem~\ref{introcharacterization}. A consequence of Theorem~\ref{foldingcharacterizationintro} is the following.  Let $G$ be a bipartite graph that is not $(p,q)$-mixing.  Since $G$ is bipartite, it folds to $K_2$.  This folding produces a sequence of graphs starting at $G$ and ending at $K_2$.  In the sequence there is a graph $H$ such that $H$ and all its predecessors are not $(p,q)$-mixing while all successors of $H$ are $(p,q)$-mixing.  In the special case of $C_{2k+1}$-mixing, we prove such a sequence can be found where $H = C_{4k+2}$.   From this characterization, Theorem~\ref{RickJonconjecture} follows easily. 

In Section~\ref{planarsection} we prove Theorem~\ref{introplanarity}. The general approach is similar to the approach taken for \Hmix{3} in~\cite{Mixing3Col}. We first argue that our graph can be assumed to be $2$-connected. Next, we argue for any $2 < \frac{p}{q} < 4$, if the graph has 
a planar embedding with at most one large face (large depends on $p$ and $q$), then it is $(p,q)$-mixing.  The converse holds in the absence of short separating cycles (short depends on $p$ and $q$); however, the general situation seems complicated.  When $3 \leq \frac{p}{q} <4$, we show that we can decompose the graph into subgraphs with no short separating cycles (specifically $C_4$) and apply the characterization to each subgraph.
The polynomial time algorithm follows easily from this, as the blocks can can be found in polynomial time as can all separating 4-cycles.  Once decomposed, listing the face lengths of a planar embedding can be done in polynomial time.
By contrast we give some examples when $2 < \frac{p}{q} < 3$ where we cannot decompose our graph to have no small separating cycles, showing that the situation is more complex.

\section{Characterizing Mixing by Cycle Winds}
\label{windingcharacterizationsection}

In this section, for $2 \leq \frac{p}{q} < 4$, we characterize graphs that are $(p,q)$-mixing. As noted above, for $\frac{p}{q} = 2$, \Hmix{G_{p,q}} is trivially in P since a graph $G$ is $(p,q)$-mixing if and only if $G$ has no edges. 

Our characterization of $(p,q)$-mixing builds on the following characterization of $(p,q)$-colouring reconfiguration.  Given a $(p,q)$-colouring $f$ of $G$,  a vertex $v$ is \emph{fixed under $f$} if for all $(p,q)$-colourings $g$ to which $f$ reconfigures, we have $g(v) = f(v)$.  
\begin{thm}[\cite{circularReconfig}]
\label{Path lemma}
Fix $p,q \in \mathbb{N}$ such that $2 \leq \frac{p}{q} <4$. Let $G$ be a graph and let $f, g$ be two $(p,q)$-colourings of $G$. Then $f$ reconfigures to $g$ if and only if the following three conditions are satisfied:
\begin{itemize}

\item{For all vertices $v \in V(G)$, if $v$ is fixed under $f$, then $v$ is fixed under $g$. Furthermore, for all fixed vertices $v$, we have $f(v) = g(v)$.}

\item {For all cycles $C$ in $G$, $W(C,f) = W(C,g)$.}

\item {For all paths $P$ in $G$ whose endpoints are fixed, $W(P,f) = W(P,g)$.}
\end{itemize}
\end{thm}

Thus, to certify that $G$ is not $(p,q)$-mixing, it suffices to find two colourings that violate one of the three conditions above.  However, we now show for the mixing problem this reduces to checking the weight of cycles.
Indeed if $G$ admits a $(p,q)$-colouring $f$ with fixed vertices, then $G$ is not $(p,q)$-mixing. To see this, suppose $v$ is a fixed vertex under $f$. Define $g$ such that for all $x \in V(G)$, $g(x) = (f(x) +1) \bmod p$. Then $g$ is a $(p,q)$-colouring, since for any edge $xy \in E(G)$, we have $|g(x) -g(y)| = |f(x)-f(y)|$.  As $f(v) \neq g(v)$, $f$ does not reconfigure to $g$.  

We now show that the existence of fixed vertices under some colouring $f$ implies the existence of two colourings and a cycle $C$ on which the colourings have different weights. 

Given a graph $G$ and a $(p,q)$-colouring $f$, a vertex $v \in V(G)$ is \emph{locked} if for all $(p,q)$-colourings $g$ which are adjacent to $f$ in $\Col(G,G_{p,q})$ satisfy $f(v) = g(v)$. (That is, $v$ cannot change colour until some other vertex does first.)

\begin{lem}[\cite{circularReconfig}]
\label{lockedlemma}
Let $2 < \frac{p}{q} < 4$.  Suppose $G$ is a graph and $f$ is a $(p,q)$-colouring of $G$. Then a vertex $v$ is locked if and only if there exist vertices $u, w \in N(v)$ such that $W(uv,f)=W(vw,f)=q$ or $W(uv,f)=W(vw,f)=p-q$.
\end{lem}

\begin{cor}[\cite{circularReconfig}]
\label{fixedcycle}
Let $2 < \frac{p}{q} < 4$. Given a graph $G$ and a $(p,q)$-colouring $f$ of $G$, suppose that $c_0$ is a fixed vertex under $f$. Then there exists a cycle $C$, $V(C) = c_0, c_1, \ldots,c _{l-1}, c_0$ such that $W(c_i c_{i+1},f) = q$ for each $i$ or $W(c_i c_{i+1},f) = p-q$ for each $i$. Here the subscripts on the vertices are taken modulo $l$.  
\end{cor} 

Whereas, fixed vertices play a crucial role in the the $\Hrec{G_{p,q}}$ problem, we now show that, for $2 < \frac{p}{q} <4$, testing if $G$ is $(p,q)$-mixing reduces to the single condition of checking the wind of cycles. 

\begin{lem}
\label{weightsoncyclelemma}
Let $2< \frac{p}{q} <4$ and $G$ be a graph. The graph $G$ is not $(p,q)$-mixing if and only if 
there is a cycle $C$ in $G$ and two colourings $f$ and $g$ where $W(C,f) \neq W(C,g)$. 
\end{lem}

\begin{proof}

If there are two colourings $f$ and $g$ where $W(C,f) \neq W(C,g)$, 
then by Theorem~\ref{Path lemma} $f$ does not reconfigure to $g$ and thus 
$G$ is not $(p,q)$-mixing.

Now suppose $G$ is not $(p,q)$-mixing. Suppose $f$ and $f'$ are in different 
components of $\Col(G,G_{p,q})$. If neither $f$ or $f'$ has fixed vertices, 
then by Theorem~\ref{Path lemma} there is a cycle $C$ such that 
$W(C,f) \neq W(C,f')$. Therefore without loss of generality $f$ has a fixed vertex. 
By Corollary~\ref{fixedcycle} we can find a 
cycle $C = c_0, c_1, \dots c_{l-1}, c_0$ such that $W(c_i c_{i+1}, f) = q$ 
for all $i$ or $W(c_i c_{i+1}, f) = p-q$ for all $i$.  
Assume the former holds (the latter is analogous).  Thus, $W(C,f) = q|E(C)|$. Define a 
$(p,q)$-colouring $g$ by $g(v) = p-f(v)$ for all $v \in V(G)$. (If $f(v) = 0$, then $g(v) = p$
and we reduce $g(v) \bmod{p}$ so that $g(v) = 0$.)
This is a $(p,q)$-colouring, since for any edge $uv \in E(G)$,  
we have $|g(v) - g(u)| = |p- f(v) - p + f(u)| = |f(u) - f(v)|$. 
(In the case $f(v)=g(v)=0$, a similar analysis verifies $g(u)g(v)$ is an edge of $G_{p,q}$.)
In particular $W(c_i c_{i+1}, g) = p-q$ for all $i$ as $-q \bmod{p} = p-q$.
Therefore $W(C,g) = (p-q)|E(C)|$.  Since $\frac{p}{q} \neq 2$, $W(C,g) \neq W(C,f)$.
\end{proof}

We now state our characterization for $(p,q)$-mixing.

\begin{thm}
\label{badweight}
Fix $2 < \frac{p}{q} < 4$ and let $G$ be a graph. Then $G$ is not $(p,q)$-mixing if and only 
if there exists a $(p,q)$-colouring $f$ of $G$ and a cycle $C$ in $G$ where 
$W(C,f) \neq \frac{|E(C)|}{2}p$.
\end{thm}

\begin{proof}
If $G$ is not $(p,q)$-mixing, then by Lemma~\ref{weightsoncyclelemma} 
there exist $(p,q)$-colourings $f$ and $g$, and a cycle $C$ such that 
$W(C,f) \neq W(C,g)$. Then at most one of $W(C,f)$ or 
$W(C,g)$ equals $\frac{|E(C)|}{2}p$, so the result follows.

Conversely, assume that we have a $(p,q)$-colouring $f$ and cycle 
$C=c_0, c_1, \dots, c_l, c_0$ where $W(C,f) \neq \frac{|E(C)|}{2}p$. 
Consider the $(p,q)$-colouring $g$ where $g(v) = p -f(v)$ for all $v \in V(G)$
(with $g(v)=0$ when $f(v)=0$). Then $W(c_i c_{i+1},g) = p - W(c_i c_{i+1},f)$.
Thus, $W(C,g) = |E(C)|p - W(C,f)$.  
Since $W(C,f) \neq \frac{|E(C)|}{2}p$, we have 
$W(C,g) \neq W(C,f)$ and thus $G$ is not $(p,q)$-mixing.
\end{proof}

The following is immediate.
\begin{cor}
The $\Hmix{G_{p,q}}$ problem is in co-NP when $2 \leq \frac{p}{q} <4$. 
\end{cor}

It follows from Theorem~\ref{badweight} that if a graph $G$ is $(p,q)$-mixing, then $G$ is bipartite. This observation had already been made in~\cite{BrewsterNoel}. In fact, they proved a stronger statement which we need some definitions to state. Define the \emph{circular mixing number} as $$m_{c}(G) = \inf\{\frac{p}{q} \ | \ G \text{ is } G_{p,q}\operatorname{-\text{mixing}}\}.$$ Let $\omega(G)$ be the \emph{clique number} of $G$, that is the number of vertices in the largest clique of $G$. 

\begin{thm}[\cite{BrewsterNoel}]
\label{bound}
Let $G$ be a non-bipartite graph. Then $m_{c}(G) \geq \max\{4, \omega(G) +1\}$. 
\end{thm}

As we are interested in $(p,q)$-mixing when $2 < \frac{p}{q} < 4$, we will restrict our attention to bipartite graphs from now on.

Let $f$ be a $(p,q)$-colouring of $G$. For any cycle $C$ of $G$, if $W(C,f) \neq \frac{|E(C)|}{2}p$, then we say that $C$ is \textit{wrapped under $f$}. To explain this choice of terminology, consider a $C = c_{0},c_{1},\ldots,c_{2l-1}, c_0$ and a colouring $f$ where $W(C,f) \neq \frac{|E(C)|}{2}p$.  The colouring $g$ defined by $g(c_i)=0$ for $i$ even and $g(c_i)=q$ for $i$ odd has $g(C) = 0q$, a single edge, and $W(C,g) = \frac{|E(C)|}{2}p$.  Hence $f$ does not reconfigure to $g$.  Consequently, any colouring to which $f$ reconfigures must have a cycle (of $G_{p,q}$) in its image (of $C$), i.e. any colouring to which $f$ reconfigures must wrap $C$ around a cycle of $G_{p,q}$.  Alternatively, a colouring of a cycle can be reconfigured to a colouring where the image is a single edge if and only if the cycle is not wrapped.  Thus, even cycles may or may not be wrapped but odd cycles are always wrapped.

We complete this section with a final winding number result. 
The \emph{sum of two cycles $C_1$ and $C_2$} is the graph induced by the edges of the 
symmetric difference $E(C_1) \triangle E(C_2)$. Given a wrapped cycle
that is the sum of two other cycles, then one of the two cycles is also wrapped. 
Consequently, given a wrapped cycle, we can find a chordless cycle that is wrapped.

\begin{lem}
\label{chordlesscycle}
Let $G$ be a graph and $f$ a $(p,q)$-colouring of $G$. 
Let $C=c_0, c_1, \dots, c_l, c_0$ be a cycle which is wrapped under $f$. 
Let $P=p_0, p_1, \dots, p_k$ be a path whose endpoints 
lie on $C$, i.e. $p_0 = c_s$ and $p_k = c_t$, and whose internal vertices do not lie on $C$. 
Then either $C'=c_s, c_{s+1}, \dots, c_{t-1}, c_{t}, p_{k-1}, \dots, p_1, c_s$ or 
$C''=c_s, p_1, \dots, p_{k-1}, c_t, c_{t+1}, \dots, c_{s-1}, c_s$ is wrapped under $f$.
\end{lem}

\begin{proof}
Observe $W(C',f) + W(C'',f) = W(C,f) + \sum_{i=1}^k (W(p_{i-1}p_i,f) + W(p_{i}p_{i-1},f)) = W(C,f)+ kp$. The reduction of the summation follows from Observation~\ref{obs:negedge}. Since $|E(C')| = t-s+k$ and $|E(C'')|=l-t+s+k$, if neither is wrapped, then $W(C',f) + W(C'',f) = (t-s+k)p/2 + (l-t+s+k)p/2 = lp/2 + kp$. This implies $W(C,f)= lp/2$, i.e. $C$ is not wrapped.  The result follows.
\end{proof}

\section{Mixing and the folding operation}
\label{foldingsection}
In this section we prove Theorem \ref{foldingcharacterizationintro}. For ease, we recall the definition of folding. Given a graph $G$, and two vertices $x$ and $y$ at distance 2, let $G_{xy}$ be the graph obtained by identifying $x$ and $y$ and call the new vertex $v_{xy}$.  The homomorphism $f: G \to G_{xy}$ defined by $f(x) = f(y) = v_{xy}$ and $f(u)=u$ for $u \not\in \{x,y\}$ is an \emph{elementary fold}. We say a graph $G$ \emph{folds} to a graph $H$ if there is a homomorphism $f: G \to H$ where $f$ is a composition of elementary folds and $f(G)$ is isomorphic to $H$. We call such a mapping a \emph{folding}. For brevity, we may refer to an elementary fold as just a fold.

\subsection{$\mathbf{(p,q)}$-mixing is closed under folding}

\begin{lem}
\label{foldsto}
Let $H$ be a graph which is not $(p,q)$-mixing where $2 < \frac{p}{q} <4$. If $G$ folds to $H$, 
then $G$ is not $(p,q)$-mixing. 
\end{lem}

\begin{proof}
It suffices to prove the statement for a single elementary fold.  Suppose $f: G \to G_{xy}$ 
is an elementary fold and $G_{xy}$ is not $(p,q)$-mixing.  Then by Theorem~\ref{badweight} 
there is a $(p,q)$-colouring of $G_{xy}$, $\psi$, and a cycle $C = v_{0},v_{1},\ldots,v_{t},v_{0}$ 
such that $C$ is wrapped under $\psi$.  The homomorphism $\psi' = \psi \circ f$ 
is a $(p,q)$-colouring of $G$.  We will show there is a cycle $C'$ in $G$ that is 
wrapped under $\psi'$.

Let $v_{xy}$ be the vertex formed by the identification of $x$ and $y$ under $f$.  
If $v_{xy} \not\in V(C)$, then $C$ is a cycle of $G$ and is wrapped under $\psi'$ as 
$\psi = \psi'$ on $C$.  

Thus, assume without loss of generality $v_{xy} = v_1$.  In $G$, if either $v_0, v_2 \in N(x)$ 
or $v_0, v_2 \in N(y)$, then with a slight abuse of notation $C$ belongs to $G$ and is wrapped 
under $\psi'$.  Otherwise, consider a common neighbour of $x$ and $y$, say $a$.  
By Lemma~\ref{chordlesscycle}, we may assume $C$ is an induced cycle of
$G_{xy}$ which implies (without loss of generality) the cycle 
$C': v_0, x, a, y, v_2, v_3, \dots, v_t, v_0$ is a cycle of $G$. Since 
$\psi'(x) = \psi'(y) = \psi(v_{xy})$ we have by Observation~\ref{obs:negedge},

\[W(xa,\psi') + W(ay,\psi') = (\psi'(a) - \psi'(x)) \bmod p + (\psi'(y) - \psi'(a)) \bmod p = p.\]
Hence,
\begin{align*}
 W(C',\psi') &= W(C,\psi)+ W(xa,\psi')+W(ay,\psi') \\
 & = W(C,\psi) + p \\
 & \neq \frac{|E(C)|}{2}p + p \\
 & = \frac{|E(C')|}{2}p.
\end{align*}
The result follows.
\end{proof}

\subsection{$\mathbf{C_{2k+1}}$-Mixing and Folds}
In this subsection, we show that for bipartite graphs $G$, 
$G$ is $C_{2k+1}$-mixing if and only if $G$ 
does not fold to $C_{4k+2}$.  First we characterize which cycles are $C_{2k+1}$-mixing.

\begin{obs}[\cite{BrewsterNoel}]
\label{4cyclelemma}
The cycle $C_{4}$ is $(p,q)$-mixing for all $\frac{p}{q} >2$. 
\end{obs}

\begin{prop}[\cite{BrewsterNoel}]
\label{cycleMix}
For $r \geq 3$, the even cycle $C_{2r}$ is $C_{2k+1}$-mixing for all $r \leq 2k$.
\end{prop}

\begin{obs}
\label{cyclefolds}
The even (respectively odd) cycle $C_{2k}$ ($C_{2k+1}$) folds to all even (odd) cycles 
$C_{2k'}$ ($C_{2k'+1}$) for $2 \leq k' \leq k$.
\end{obs}

\begin{proof}
Suppose the cycle is $v_0, v_1, v_2, v_3, \dots, v_t, v_0$.  Identify $v_0$ and $v_2$, and 
then $v_1$ and $v_3$. The result is a cycle with two fewer vertices, and so by induction 
the result follows.
\end{proof}

Combining these results we obtain the following.

\begin{obs}
\label{cyclesmixing}
Fix an integer $k \geq 1$. Then $C_{2r}$ is $C_{2k+1}$-mixing if and only if $r \leq 2k$.
\end{obs}

\begin{proof}
If $r \leq 2k$, then the observation follows from Proposition~\ref{cycleMix} for cycles larger than $C_{4}$, and follows from Observation~\ref{4cyclelemma} for $4$-cycles.

If $r \geq 2k+1$,  let $C_{4k+2} = c_{0},\ldots,c_{4k+1},c_{0}$. 
The $C_{2k+1}$-colouring $f(c_{i}) = ik \bmod{2k+1}$ has $W(C_{4k+2},f) = (4k+2)(k) \neq (2k+1)(2k+1)$ and hence the cycle is wrapped.  Therefore $C_{4k+2}$ is not $C_{2k+1}$-mixing, and the result follows from Lemma~\ref{foldsto}.
\end{proof}

The above observation tells us that $C_{4k+2}$ is the minimum even cycle which is not 
$C_{2k+1}$-mixing. It turns out, $C_{4k+2}$ is the unique minimal (with respect to folding) 
bipartite graph which is not $C_{2k+1}$-mixing.

\begin{thm}
\label{oddcyclepinching}
A connected bipartite graph $G$ is not $C_{2k+1}$-mixing if and only if $G$ folds to $C_{4k+2}$.
\end{thm}

\begin{proof}
If $G$ folds to $C_{4k+2}$, then the claim follows from Lemma~\ref{foldsto} 
and Observation~\ref{cyclesmixing}.

To prove the converse, suppose $G$ is a vertex minimal counter-example.  That is, $G$ is not 
$C_{2k+1}$-mixing, but does not fold to $C_{4k+2}$.  By Theorem~\ref{badweight}, we may 
assume there is a $C_{2k+1}$-colouring $\psi$ of $G$ and $C$ a cycle of $G$ that is wrapped 
under $\psi$.  By Observations~\ref{cyclefolds} and~\ref{cyclesmixing} we may assume that 
$G$ is not a cycle.

First, if there exist three consecutive vertices $x,y,z \in V(C)$ such that $\psi(x) = \psi(z)$, 
then $xz \not \in E(G)$ and we fold $x$ and $z$ to create the vertex $v_{xz}$. 
Call the resulting graph $G_{xz}$ and the resulting cycle $C'$. The map $\psi$ induces a 
colouring $\psi': G_{xz} \to C_{2k+1}$ by
$$
\psi'(w) = \left\{ \begin{matrix} \psi(x) & \mbox{ if } w = v_{xz} \\
                                  \psi(w) & \mbox{ otherwise. } \end{matrix} \right.
$$
Let $P'$ be the path in $C$ from $z$ to $x$ not using $y$, and $P''$ be the path $x,y,z$. Then $W(C,\psi) = W(P'',\psi) + W(P',\psi)$. We have that  $W(P'',\psi) = p$ by Observation~\ref{obs:negedge}.  Finally, $W(P',\psi) = W(C',\psi')$ by construction.  Thus, $W(C',\psi') = W(C,\psi) - p$, $E(C') = E(C)-2$ and hence $W(C',\psi') \neq |E(C')|p/2$ and $C'$ is wrapped under $\psi'$. By minimality $G_{xz}$ folds to $C_{4k+2}$ and thus $G$ folds to $C_{4k+2}$, a contradiction.

Hence assume three such vertices do not exist. Since $G$ is connected, but not a cycle, there is a vertex $a \in V(G) \setminus V(C)$ adjacent to a vertex, say $y \in V(C)$.  Let $x$ and $z$ be the neighbours of $y$ in $C$. Observe that the vertex $\psi(y)$ is a vertex of the cycle $C_{2k+1}$. As $C_{2k+1}$ is $2$-regular and $\psi(x) \neq \psi(z)$, $\psi(x)$ and $\psi(z)$ are the two distinct neighbours of $\psi(y)$. As $\psi(a)$ is also a neighbour of $\psi(y)$, it follows that either $\psi(a) = \psi(x)$ or $\psi(a) = \psi(z)$. Without loss of generality we assume  $\psi(a) = \psi(x)$.  
Let $f$ be the fold identifying $a$ and $x$ to form $G_{ax}$.  The cycle $C$ is (still) 
wrapped under the $C_{2k+1}$-colouring $\psi \circ f$ of $G_{ax}$.  Consequently 
$G_{ax}$ is not $C_{2k+1}$-mixing. By minimality, $G_{ax}$ folds to $C_{4k+2}$ and thus $G$ folds to $C_{4k+2}$, a contradiction.
\end{proof}

Note that the folding operation cannot reduce the number of connected components in a graph.  
(Vertices from different components cannot be identified.)  Since a graph is $C_{2k+1}$-mixing 
if and only if each connected component is $C_{2k+1}$-mixing, we can remove the connected 
condition in the theorem and simply require that at least one component folds to $C_{4k+2}$.

We now examine foldings of non-bipartite graphs.  Using the above proof we construct a partial extension to circular colourings of the following theorem of Cook and Evans.

\begin{thm} [\cite{Chromaticnumberandfolding}]
Let $G$ be a connected graph. If $\chi(G) = k$, then $G$ folds to $K_{k}$. 
Furthermore, any graph that folds to $K_{k}$ is $k$-colourable.
\end{thm}

One might conjecture that if $\chi_{c}(G) = \frac{p}{q}$, and $G$ is connected, then $G$ folds to $G_{p,q}$. Unfortunately, this is not the case, due to a theorem of Zhu~\cite{zhu_1999}. 

\begin{thm}[\cite{zhu_1999}]
Fix integers $p,q$ such that $p \neq 2q+1$ and $q \neq 1$. Then there exists a graph $G$ such that $G$ is a strict subgraph of $G_{p,q}$, and $\chi_{c}(G) = \frac{p}{q}$. 
\end{thm}

As a strict subgraph of $G_{p,q}$ cannot possibly fold to $G_{p,q}$ we see that the natural conjecture is false.  As~\cite{Chromaticnumberandfolding} affirms $G$ does fold to the clique that determines its chromatic number, the result of Zhu does leave open the possibility that $G$ folds to the odd cycle $G_{2k+1,k}$ when $\chi_c(G) = \frac{2k+1}{k}$.  We now prove this is the case.

We note that the hypothesis  in Theorem~\ref{oddcyclepinching} requiring the graph to be bipartite is used only to determine the shortest cycle to which $G$ can fold and not be $C_{2k+1}$-mixing.  In fact, 
the proof of Theorem~\ref{oddcyclepinching} shows for any graph $G$,  if $G$ admits a 
$C_{2k+1}$-colouring $\psi$ and contains a wrapped cycle $C$ under $\psi$, 
then $G$ folds to some cycle $C'$ (and $C'$ admits a $C_{2k+1}$-colouring).    
In particular, if $G$ is not bipartite (and thus $G$ has a wrapped cycle under any $\psi$), then  
$G$ folds to an odd cycle $C'$ and $C'$ has length at least $2k+1$. Hence we have the following lemma.

\begin{lem}
\label{foldingwhennotbipartite}
Let $k \geq 1$ be an integer.  Suppose $G$ is a connected, non-bipartite graph and has a $C_{2k+1}$-colouring. Then $G$ folds to $C_{2k+1}$. 
\end{lem}

\begin{cor}
Let $G$ be a connected graph. If $\chi_{c}(G) = \frac{2k+1}{k}$, then $G$ folds to $C_{2k+1}$. 
Furthermore, any graph which folds to $C_{2k+1}$ is $C_{2k+1}$-colourable.
\end{cor}

\begin{proof}
Suppose $\chi_{c}(G) = \frac{2k+1}{k}$. Then $G$ is not bipartite, and hence Theorem~\ref{bound} implies that $G$ is not $C_{2k+1}$-mixing, as $4 > \frac{2k+1}{k}$. 
By Lemma~\ref{foldingwhennotbipartite}, $G$ folds to $C_{2k+1}$. To see the second claim, 
observe that elementary folds are homomorphisms and homomorphisms compose.  
\end{proof}

Combining the above discussion gives the following result.
\begin{thm}
If $q = 1$ or $p = 2q+1$, then every connected graph $G$ with $\chi_{c}(G) = \frac{p}{q}$ folds to $G_{p,q}$. For all other pairs $(p,q)$, there exists a connected graph $G$ with $\chi_{c}(G) = \frac{p}{q}$ that does not fold to $G_{p,q}$. 
\end{thm}

 We finish this subsection by showing that Theorem~\ref{oddcyclepinching} that the circular mixing number of bipartite graphs is $2$ as conjectured in~\cite{BrewsterNoel}.

\begin{lem}
For every connected, bipartite graph $G$, there exists a $k_{0}$ such that for every $k \geq k_{0}$, $G$ is $C_{2k+1}$-mixing.
\end{lem}

\begin{proof}
Let $C$ be a largest cycle in $G$. Pick $k_{0}$ such that $4k_{0}+2>|V(C)|$ and $k \geq k_0$. By Lemma~\ref{oddcyclepinching}, we know that $G$ is not $C_{2k+1}$-mixing if and only if $G$ folds to $C_{4k+2}$. Consider any cycle of $G$ and any elementary folding applied to $G$.  Then the length of the cycle remains the same or is reduced by $2$.  Consequently, folding $G$ produces a graph where all cycles have length strictly less than $4k_0+2$.  In particular, $G$ cannot fold to $C_{4k+2}$.  Therefore, $G$ is $C_{2k+1}$-mixing.
\end{proof}

Note that the choice of $k_{0}$ given in the above proof is tight, and can be seen by considering $C_{4k+2}$. Since $\frac{2k+1}{k}$ tends to $2$ as $k$ approaches infinity, we have:

\begin{cor}
If $G$ is bipartite, then $m_{c}(G) =2$.
\end{cor}

Combining this with Theorem \ref{bound} we conclude:

\begin{thm}
\label{completebound}
If $G$ is bipartite, then $m_{c}(G)=2$. Otherwise, $m_{c}(G) \geq \max{\{4,\omega(G)+1\}}$.
\end{thm}

\subsection{Showing that odd cycle mixing is co-NP-complete}

In this subsection we show the problem of determining whether an input graph $G$ folds
to the cycle $C_{2k}$, denoted as the \Hfold{C_{2k}} problem, is NP-complete.
For our reduction we require two additional homomorphism problems.   

\begin{defn}
Let $G$ be a graph and $H$ an induced subgraph of $G$. A \emph{retraction} of $G$ to $H$ is a 
homomorphism $r: G \to H$ such that $r(h) =h$ for every vertex $h \in V(H)$. 
If there exists a retraction of $G$ to $H$, we say $G$ \emph{retracts} to $H$.
\end{defn}

The \Hret{H} problem takes as input a graph $G$ containing a labelled subgraph isomorphic to $H$.
It asks if there is a retraction of $G$ to the copy of $H$ in $G$.  We will give a reduction 
from \Hret{C_{2k}}, which is known to be NP-complete for all $k \geq 3$ 
(proven independently by Feder, and by Dukes, Emerson, and MacGillivray, see \cite{vikas, rettocycles}), 
to the \Hfold{C_{2k}} problem. 

An easy observation is that for connected graphs, a retraction is a fold.  (Consider a vertex
$v$ at distance one from $H$.  The image of $v$ under the retraction is at distance two from $v$.) 
Thus, we have the following.

\begin{obs}[\cite{Hahn1997,godsil01}] 
\label{retracttofold}
Let $G$ be a connected graph, and $H$ an induced subgraph of $G$. If $G$ retracts to $H$, 
then $G$ folds to $H$.
\end{obs}

We also require a theorem of Vikas~\cite{vikas} relating retractions and compactions.

\begin{defn}
A \emph{compaction} of $G$ to $H$ is a homomorphism $G \to H$ that is onto
$V(H)$ and onto all non-loops of $E(H)$.  We say $G$ \emph{compacts} to $H$.
\end{defn}

The problem \Hcomp{H} takes an input graph $G$ and asks if there is a compaction of
$G$ to $H$. It is easy to see for graphs without isolated vertices, a homomorphism that is edge surjective
is also vertex surjective.  Additionally for irreflexive graphs the homomorphism $f: G \to H$ is a 
compaction if and only if $f(G) = H$.

\begin{obs}
\label{foldtocompact}
Suppose $G$ and $H$ are irreflexive graphs.
Let $f:G \to H$ be a folding. Then $f$ is a compaction of $G$ to $H$. 
\end{obs}

\begin{proof}
By definition of a folding, $f(G) = H$ making $f$ both vertex and edge surjective.
\end{proof}

In the proof that $C_{2k}$-compaction is NP-complete for all 
$k \geq 3$, Vikas proves the following:

\begin{thm}[\cite{vikas}]
\label{Keytheorem}
Let $G$ be a graph with an induced copy of $C_{2k}$, $k \geq 3$. 
Then there is a graph $G'$ such that the following are equivalent:

\begin{itemize}
\item{$G$ retracts to $C_{2k}$,}
\item{$G'$ retracts to $C_{2k}$,}
\item{$G'$ compacts to $C_{2k}$.}
\end{itemize}
\end{thm}

We make two comments about the above theorem.  First,
the graph $G'$ in the above theorem depends on the value of $k$. 
Secondly, using the following observation, we may restrict our attention 
in the \Hret{C_{2k}} problem to connected graphs $G$.

\begin{obs}\label{obs:connected}
Let $G$ be a graph containing a labelled copy of $C_{2k}$ as an induced subgraph of $G$. 
Then $G$ retracts to the copy of $C_{2k}$ if and only if the connected component containing 
$C_{2k}$ retracts to $C_{2k}$ and all connected components are bipartite.
\end{obs}

\begin{proof}
The necessary condition is obvious. Conversely, mapping all connected components not containing 
$C_{2k}$ to an edge of $C_{2k}$ (which is possible since they are bipartite) and using the 
retraction of the connected component containing $C_{2k}$ gives a retraction of $G$ to
the labelled copy of $C_{2k}$. 
\end{proof}

In Vikas's proof of Theorem~\ref{Keytheorem}, the graph $G'$ is constructed from 
$G$ by adding paths (whose lengths depend on $k$).  Thus, if $G$ is connected,
then we can conclude $G'$ is connected as well.  Further, the construction
will produce a bipartite graph $G'$ when $G$ is bipartite.
Therefore we can strengthen Vikas's result to the following:

\begin{thm}\label{thm:tfae}
Let $G$ be a connected graph with an induced copy of $C_{2k}$, $k \geq 3$. 
Then there is a connected graph $G'$ such that the following are equivalent:

\begin{itemize}
\item{$G$ retracts to $C_{2k}$,}

\item{$G'$ retracts to $C_{2k}$,}

\item{$G'$ folds to $C_{2k}$,}

\item{$G'$ compacts to $C_{2k}$.}
\end{itemize}
\end{thm}

\begin{proof}
First assume $G'$ retracts to $C_{2k}$.  By Observation~\ref{retracttofold}, 
since $G'$ retracts to $C_{2k}$ and $G'$ is connected, $G'$ folds to $C_{2k}$. 
Now assume that $G'$ folds to $C_{2k}$. Then by Observation~\ref{foldtocompact}, 
$G'$ compacts to $C_{2k}$. By Theorem~\ref{Keytheorem}, the result follows. 
\end{proof}

We now prove our main result for this section:

\begin{thm}
\label{conpcomplete}
For a fixed $k \geq 1$, \Hmix{C_{2k+1}} is co-NP-complete.
\end{thm}

\begin{proof}
Let $G$ be an instance of \Hret{C_{4k+2}}. Then we may assume $G$ is bipartite. 
Further by Observation~\ref{obs:connected}, we may assume $G$ is connected.  
By Theorem~\ref{thm:tfae}, there is a connected, bipartite graph $G'$ that folds 
to $C_{4k+2}$ if and only if $G$ retracts to $C_{4k+2}$.  Thus, by Theorem~\ref{oddcyclepinching} 
we have $G'$ is no instance of \Hmix{C_{4k+2}} if and only if $G$ is a yes instance of
\Hret{C_{4k+2}}.  The result follows.
\end{proof}


\section{$(p,q)$-mixing in Planar Graphs}
\label{planarsection}
Now we turn our attention to $(p,q)$-mixing in planar graphs. In~\cite{Mixing3Col}, 
Cereceda et al.\ showed that there is a polynomial time algorithm to determine 
if a graph is $3$-mixing when restricted to planar (bipartite) graphs. 
Here we extend the result by giving a polynomial time algorithm for determining if
a bipartite planar graph is $(p,q)$-mixing for 
$3 \leq \frac{p}{q} < 4$.  In~\cite{Mixing3Col}, the key ideas for the polynomial algorithm are as follows.  For $k=3$, $C_4$ is $3$-mixing and $C_6$ is not.  Let $G$ be a $2$-connected, bipartite graph with a planar embedding.  If $G$ has a separating $C_4$, then $G$ can be decomposed into two subgraphs such that $G$ is $3$-mixing if and only if both subgraphs are.  If $G$ has no separating $C_4$, then $G$ is $3$-mixing if and only if $G$ has at most one face of length at least $6$.  In extending the results to general $2 < \frac{p}{q} < 4$, we need to both understand what a large face is and how to decompose using short separating cycles. 

Our following structural result characterizes $(p,q)$-mixing for bipartite graphs with planar embeddings not containing short separating cycles.  As in~\cite{Mixing3Col}, the polynomial algorithm follows from the characterization together with a reduction of the general case to the case without short separating cycles.  
Our structural result applies to all $2 < \frac{p}{q} < 4$ without small separating cycles, but the reduction to remove separating cycles as in~\cite{Mixing3Col}, requires $3 \leq \frac{p}{q} < 4$ where the short separating cycles are precisely $C_4$.  Hence our algorithm applies only to $3 \leq \frac{p}{q} < 4$.

As a notation convention, in this section given an embedding of a planar graph $G$ we will use $G$ to denote both the graph and its embedding.

\begin{thm}\label{thm:bipartiteplanarmixing}
Fix $2 < \frac{p}{q} < 4$. Let $C_{2k}$ be the minimal non-$(p,q)$-mixing even cycle. Let $G$ be a $2$-connected bipartite planar graph.  Suppose $G$ has a planar embedding with no separating $C_{2i}$-cycles for $i \in \{2,\ldots,k-1\}$. Then $G$ is $(p,q)$-mixing if and only if $G$ has at most one face with length greater than or equal to $2k$. 
\end{thm}

Our first point of order is to show that the minimal even cycle which is not $(p,q)$-mixing is well defined.

\begin{obs}
Fix integers $p$ and $q$ such that $2 < \frac{p}{q} < 4$. If $p$ is even, then for all even integers $j \geq p$,  $C_{j}$ is not $(p,q)$-mixing. If $p$ is odd, then for all even $j \geq 2p$,  $C_{j}$ is not $(p,q)$-mixing.
\end{obs}

\begin{proof}
First assume that $p$ is even. Let $C_{p} = c_{0},\ldots,c_{p-1},c_{0}$ and consider the 
$(p,q)$-colouring $f(c_{i}) = iq \bmod{p}$, $i=0, 1, \dots, p-1$. Then $W(C,f) = qp < \frac{p}{2}p$ as $\frac{p}{q} >2$. Hence $C_{p}$ is not $(p,q)$-mixing. It follows by Lemma~\ref{foldsto} that when $p$ is even, for all even $j  \geq p$, $C_{j}$ is not $(p,q)$-mixing.

Now assume that $p$ is odd. Let $C_{2p} = c_{0},\ldots,c_{2p-1},c_{0}$. For $i = 0, 1, \ldots,p-1$, colour $c_{i}$ and $c_{i+p}$ with $iq \bmod p$. Again, $W(C,f) = 2qp < p^{2}$, and thus $C_{2p}$ is not $(p,q)$-mixing. Hence, for all even $j \geq 2p$, $C_{j}$ is not $(p,q)$-mixing. 
\end{proof}

It follows that the minimal even cycle which is not $(p,q)$-mixing is well defined. In fact, the minimum can be found in fixed time (for fixed $p$ and $q$), as we only have a fixed number of cycles which we need to check and we can check if a single cycle is $(p,q)$-mixing in fixed time. 
 
First, we reduce the problem \Hmix{G_{p,q}} problem to $2$-connected graphs. 
\begin{obs}
\label{blockreduction}
Fix integers $p$ and $q$ such that $2 < \frac{p}{q} < 4$.  Let $G$ be a graph with a $(p,q)$-colouring.  Then $G$ is $(p,q)$-mixing if and only if every block of $G$ is $(p,q)$-mixing.
\end{obs}

\begin{proof}
Suppose that a block $B$ of $G$ is not $(p,q)$-mixing. Then there is a $(p,q)$-colouring $f$ of $B$ such that some cycle $C$ in $B$ is wrapped under $f$. Observe that we can extend $f$ to a $(p,q)$-colouring of $G$ by taking $(p,q)$-colourings of each block (which exist as $G$ has a $(p,q)$-colouring), and permuting colours if necessary. 

Conversely, suppose that $G$ is not $(p,q)$-mixing. Then there is a $(p,q)$-colouring $f$ of $G$ and a cycle $C$ in $G$ which is wrapped under $f$. As a cycle is $2$-connected, $C$ lies in some block $B$. Then $f$ restricted to $B$ is a $(p,q)$-colouring which has a wrapped cycle, and hence $B$ is not $(p,q)$-mixing. 
\end{proof}

Therefore, we can assume for the remainder of this section that all graphs are $2$-connected. Recall that if a planar graph is $2$-connected, then every face is bounded by a cycle.  The next observation allows us to assume $G$ does not have a clique cutset of size two.

\begin{obs}
\label{nocliquecutsets2}
Fix integers $p$ and $q$ such that $2 < \frac{p}{q} < 4$. Let $G$ be a 2-connected graph, and suppose $uv \in E(G)$ such that $G \setminus \{u, v\}$ is disconnected with components $T_1, T_2, \dots, T_k$. Let $G_{1}$ and $G_{2}$ be induced subgraphs of $G$ such that $V(G_{1}) = V(T_1) \cup \{u,v\}$,  $V(G_{2}) = V(T_2) \cup \cdots \cup V(T_k) \cup \{u,v\}$, and $E(G_{1}) \cup E(G_{2}) = E(G)$. If $G_{1}$ and $G_{2}$ are $(p,q)$-mixing, then $G$ is $(p,q)$-mixing.
\end{obs}

\begin{proof}
Suppose to the contrary $G$ is not $(p,q)$-mixing. Then there is a $(p,q)$-colouring $f$ of $G$ and a cycle $C$ such that $C$ is wrapped under $f$. As $G_{1}$ and $G_{2}$ are $(p,q)$-mixing, this implies that $V(C) \not \subseteq V(G_{1})$ and $V(C) \not \subseteq V(G_{2})$. In particular, $u,v \in V(C)$ and $uv \not \in E(C)$. Let $C_{1}$ be the $(u,v)$-path of $C$ in $G_{1}$, and $C_{2}$ be the $(v,u)$-path of $C$ in $G_{2}$. Then $C_{1} +uv$ and $C_{2} + vu$ are cycles, and by Lemma~\ref{chordlesscycle} at least one of them is wrapped under $f$. But this contradicts the assumption that $G_{1}$ and $G_{2}$ are $(p,q)$-mixing.
\end{proof}

We now prove the necessity of the condition in Theorem~\ref{thm:bipartiteplanarmixing}.  Note we for this direction the absence of small separating cycles is not required.
We say a \textit{$\geq k$-face} is a face whose boundary has at least $k$ edges.

\begin{lem}\label{lem:atmostonelongface}
Fix $2 < \frac{p}{q} <4$. Let $C_{2k}$ be the minimal even cycle which is not $(p,q)$-mixing. Let $G$ be a $2$-connected planar bipartite graph. Suppose $G$ has a planar embedding in which there is at most one $\geq 2k$-face. Then $G$ is $(p,q)$-mixing. 
\end{lem}
 
 \begin{proof}
Let $G$ be an edge minimal counterexample to the claim. If $G$ is a cycle, then the inner face and outer face of the cycle are both faces of length equal to that of the cycle.  Hence, the cycle has length strictly less than $2k$ (there is at most one $\geq 2k$-face), and $G$ is $(p,q)$-mixing.  Thus, we assume that $G$ is not isomorphic to a cycle.

If one exists, let $f$ be the $\geq 2k$-face, otherwise let $f$ be an arbitrary face, and let $C = v_{0},\ldots,v_{t-1},v_{0}$ be the facial cycle of $f$.  Without loss of generality, let $f$ be the outer face. Thus, all faces in the interior of $C$ have length less than $2k$. If $C$ has a chord, then this chord lies in the interior of $C$, and we have  a clique cutset of size two which decomposes $G$ into two graphs $G_{1}$ and $G_{2}$, each of which only has at most one $\geq 2k$-face, since all interior faces have length strictly less than $2k$. By minimality both $G_{1}$ and $G_{2}$ are $(p,q)$-mixing, which by Observation~\ref{nocliquecutsets2} implies $G$ is $(p,q)$-mixing.

\begin{claim}
\label{consecutiveclaim}
There exists a face $f'$, where $f'$ has a facial cycle $C'$ such that $V(C') \cap V(C)$ induces a path of length at least $1$. 
\end{claim}

\begin{proof}\claimproof

Pick an arbitrary edge $e= v_{j}v_{j+1} \in E(C)$. Let $f'$ and $f$ be the two distinct faces whose boundaries contain $e$. Let $C'$ be the facial cycle of $f'$. 
We now examine the induced subgraph $G[V(C) \cap V(C')]$.  If it contains an edge not in $C$ (but whose endpoints are necessarily in $C$), then $C$ has a chord.  Thus, all edges belong to $C$.  Moreover, the subgraph cannot be all of $C$, otherwise $f'$ is the interior of $C$ and $G$ is a cycle. We may assume that $G[V(C) \cap V(C')]$ is a union of paths.  If it is a single path, then we are done, so assume there is more than one path.  Let $P$ be the component containing $v_{j}$ and let $v$ be an endpoint of $P$.  Let $Q'$ be the path in $C'$ starting at $v$ and ending at a vertex $u \in V(C)$ such that $V(Q') \cap V(P) = \{v\}$, and all internal vertices of $Q'$ are not in $V(C)$. Similarly, let $Q$ be the path in $C$ from $u$ to $v$ whose internal vertices are not in $V(C')$.  In particular, $Q$ and $P$ share only the vertex $v$.  Now $Q' \cup Q$ is a cycle $C''$ intersecting $C$ in the path $Q$.  If $C''$ is facial, then we are done. Otherwise, $Q$ shares an edge with $f$ and a face $f''$ in the interior of $C''$.  We repeat the argument.  As $G$ is finite, the result follows.

\end{proof}

Let $e$ be on the boundary of a face $f'$, with facial cycle $C'$, such that $S = G[V(C) \cap V(C')]$ is a path.  Let $v'$ and $v''$ be the endpoints of $S$ and $S'$ be the set of interior vertices. 

Consider $G' := G - S' - \{e\}$. (The deletion of $e$ is only required when $S' = \emptyset$.) Notice as all of the vertices of $S$ are on the boundary of $f$, we do not create two $\geq 2k$ faces. We have two cases to consider.

First suppose $G'$ is $2$-connected. Since $G$ is an edge minimal counterexample, $G'$ is $(p,q)$-mixing. Now consider any $(p,q)$-colouring $\phi$ of $G$. Observe that $\phi$ restricts to a $(p,q)$-colouring of $G'$. As $G'$ is $(p,q)$-mixing, for all cycles $D$ in $G'$, we have $W(D,\phi) = \frac{|E(D)|}{2}p$. It suffices to show that the winds of $C'$ and $C$ are $\frac{|E(C')|}{2}$ and $\frac{|E(C)|}{2}$ respectively.  As $|E(C')| < 2k$, we have $W(C',\phi) = \frac{|E(C')|}{2}p$. So it suffices to show that $W(C,\phi) = \frac{|E(C)|}{2}p$. To see this, consider the cycle $C'' := C \Delta C'$. Then since $W(C'',\phi) = \frac{|E(C'')|}{2}p$ and $W(C',\phi) = \frac{|E(C')|}{2}p$, by Lemma \ref{chordlesscycle} it follows that $W(C,\phi) = \frac{|E(C)|}{2}p$.

 Otherwise, $G'$ has a cut vertex $v$. We are going to argue that this cannot occur. Let $T_{1},T_{2},\ldots,T_{t}$ be the components of $G' - v$. We claim that $t=2$. As $v$ is a cut vertex, by definition $t \geq 2$. By construction, all vertices in $S'$ have degree $2$ in $G$. Adding the path $S$ to $G'$ joins at most two of the components in $G' -v$. Therefore if $t \geq 3$, $G$ contains a cut vertex, a contradiction. 
 
 Thus, we can decompose $G'$ into two graphs $G_{1}$ and $G_{2}$, such that $V(G_{1}) \cap V(G_{2}) = \{v\}$,  both $G_{1}$ and $G_{2}$ contain at least two vertices, and up to relabelling, $v' \in V(G_{1})$, $v'' \in V(G_{2})$. 
 
 First suppose that $v = v'$. Then as $G'$ is simply $G$ with $S'$ and $e$ deleted, there is a vertex $x \in V(G_{1}) \setminus \{ v'\}$ such that every path from $x$ to $v''$ contains $v'$. But this implies that $G$ is not $2$-connected, a contradiction. By a similar argument, we can assume that $v \neq v''$.

Now we can conclude that all paths from $v'$ to $v''$ in $G'$ must have $v$ as an internal vertex.  In particular, the path from $v''$ to $v'$ in $C-S'-\{ e \}$ must contain $v$.  We conclude $v \in V(C)$.  The boundary of $f'$ contains a path from $v'$ to $v''$ that does not use $S' \cup \{ e \}$.  In particular, this path must contain $v$.  We conclude $v \in V(C) \cap V(C')$ contrary to our assumption that $V(C) \cap V(C')$ is the path $S$ which establishes the lemma.
\end{proof}

We now prove the sufficiency of the condition in Theorem~\ref{thm:bipartiteplanarmixing}. We begin with some observations used in the proof.

\begin{obs}
\label{planefolds}
Let $G$ be a planar graph, let $x,y,z$ be consecutive vertices on a face, and suppose that $d(x,z) =2$. Then the graph obtained by folding $x$ and $z$ is planar. 
\end{obs}

One can see this by simply adding the edge $xz$ inside the face and then contracting the edge. 

The next lemma due to Hell is well known and follows from a simple breadth first search argument, see for example~\cite{Hell_1972, Graphsandhomomorphisms}.

\begin{lem}
\label{shortestpath}
Let $G$ be a connected bipartite graph, and let $P$ be a shortest path from $x$ to $y$ in $G$. Then $G$ retracts to $P$. 
\end{lem}

We state a useful and well known corollary of this fact.

\begin{cor}
\label{shortestcycleretract}
Let $G$ be a connected bipartite graph, and $C$ a shortest cycle in $G$. Then $G$ retracts to $C$.
\end{cor}
\begin{proof}
Let $e$ be any edge in $C$. Then $G-e$ retracts to $C-e$ by Lemma \ref{shortestpath}, and thus $G$ retracts to $C$.
\end{proof}

Let $G$ be a connected graph and $C$ be a cycle in $G$. Given a planar embedding of $G$, following the notation of~\cite{Mixing3Col}, we let $\text{Int}(C)$ and $\text{Ext}(C)$ be the set of vertices inside and outside $C$ respectively.  Neither of these sets include the vertices of $C$.  We define \emph{the interior of $C$}, denoted $G_{\text{int}}(C)$, to be the induced subgraph $G[V(C) \cup \text{Int}(C)]$ and \emph{the exterior of $C$} to be $G_{\text{ext}}(C) = G[V(C) \cup \text{Ext}(C)]$.  The cycle $C$ is \emph{separating} if both $\text{Int}(C)$ and $\text{Ext}(C)$ are non-empty.
We use the abbreviated notation $G_{\text{int}}$ and $G_{\text{ext}}$ when the cycle is clear from context.

We now establish the sufficiency of Theorem~\ref{thm:bipartiteplanarmixing}.

\begin{lem}\label{lem:atleasttwolongfaces}
Fix $2 < \frac{p}{q} < 4$. Let $C_{2k}$ be the smallest even cycle which is not $(p,q)$-mixing. Let $G$ be a $2$-connected bipartite graph with a planar embedding containing no separating $C_{2i}$-cycle for all $i \in \{2,\ldots, k-1\}$. If $G$ has at least two $\geq 2k$-faces, then $G$ is not $(p,q)$-mixing. 
\end{lem}

\begin{proof}
Let $f, f_{o}$ be two $ \geq 2k$-faces, and suppose that the boundaries of $f$, and $f_{o}$, are $C$, and $C_o$, respectively. Without loss of generality, we may suppose that $f_{o}$ is the outer face. 
A cycle $D$ is \emph{$f$-separating} if $f$ lies in $G_{\text{int}}(D)$ and either $D$ is a separating cycle or $D=C_o$, the boundary of the outer face.  In particular, $C_o$ is $f$-separating despite the fact that $C_o$ is not separating. As $|E(C_o)| \geq 2k$ and $G$ has no separating cycles of length less than $2k$, $G$ has no $f$-separating cycles of length less than $2k$.
  
If $G$ is a cycle, then $G$ folds to $C_{2k}$ and $G$ is not $(p,q)$-mixing by Lemma~\ref{foldsto}.  Otherwise, we show that $G$ folds to a bipartite graph $G'$ on fewer vertices, such that $G'$ contains a block with two $\geq 2k$-faces, the face $f$ (from $G$) and $f'_o$ (the outerface of $B'$), and no $f$-separating cycle of length $2i$, $i \in \{ 2, \dots, k-1 \}$.  The result follows by induction.
  
Therefore assume $G$ is not a cycle and let $y \in V(C)$ such that $\deg(y) \geq 3$.   

Let $z$ be a neighbour of $y$ on $C$, and $a$ be a neighbour of $y$ not in $C$, such that all of $a,y,$ and $z$ lie on a face (such a choice of $z$ and $a$ exists). Fold $a$ and $z$ and let $G'$ be the resulting graph. As $a,y$ and $z$ lie on a face of $G$, the graph $G'$ is planar.  
Moreover, $f$ is still a face of $G'$ as $a \not\in V(C)$. (We may think equivalently of $G'$ as being formed by deleting $a$ and joining $z$ to all the neighbours of $a$.  As $a \in \text{Ext}(C)$ this process leaves the face $f$ unchanged.)   Let $B'$ be the block of $G'$ containing $f$.  Now we consider two cases.

\smallskip\noindent  
\textbf{Case $1$: $B'$ has no $f$-separating $C_{2i}$-cycle for $i \in \{2,\ldots,k-1\}$. }

The outerface of $B'$ is an $f$-separating cycle.  Thus, it has length at least $2k$.  The outerface together with $f$ are two $\geq 2k$-faces of $B'$.  By induction, $B'$ is not $(p,q)$-mixing from which we conclude $G'$ is not $(p,q)$-mixing by Observation~\ref{blockreduction}, and by Lemma~\ref{foldsto}, $G$ is not $(p,q)$-mixing.
 
\smallskip\noindent    
\textbf{Case $2$: $B'$ has an $f$-separating cycle $D'$ of length less than $2k$.}

Let $v_{az}$ be the new vertex obtained from folding $a$ and $z$.  As $G$ has no $f$-separating cycle of length less than $2k$, we have $v_{az} \in V(D')$.  Let $v'_{az}$ and $v''_{az}$ be the two neighbours of $v_{az}$ in $D'$. Observe that without loss of generality, $a$ is adjacent to $v'_{az}$ and not to $v''_{az}$ in $G$ and $z$ is adjacent to $v''_{az}$ but not to $v'_{az}$ in $G$, as otherwise, $G$ has an $f$-separating cycle of length less than $2k$. Let $D$ be the cycle (in $G$) that gave rise to $D'$ in $B'$, i.e. $D$ is the cycle obtained by replacing the path $v'_{az}, v_{az}, v''_{az}$ in $D'$ with the path $v'_{az}, a, y, z, v''_{az}$.  Since $C$ is in $G_{\text{int}}(D)$ and $C_o$ is in $G_{\text{ext}}(D)$, if $|E(D)| < 2k$, then $G$ contains an $f$-separating cycle of length less than $2k$, a contradiction.  Since $D'$ has length less than $2k$, it follows that $|E(D)| = 2k$ and $|E(D')| = 2k-2$.

We claim $D$ is an $f$-separating cycle.  As $|E(C)| \geq 2k$, $|E(D)|=2k$, and $a \in V(D) \backslash V(C)$, there is a vertex of $C$ in $\text{Int}(D)$.  Thus, $D$ is $f$-separating if there is a vertex in $\text{Ext}(D)$ or $D = C_o$ the boundary of the outerface.  Suppose neither holds. Since $|E(C_o)| \geq 2k$ and $|E(D)| = 2k$, it must be the case that $V(C_o)= V(D)$.  Hence $D$ must have a chord not belonging to $C_o$ as $D \neq C_o$.  However, this chord must be in the interior of $D$ (as $C_o$ bounds the outerface) which implies $D$ is the sum of two shorter cycles one of which is $f$-separating, a contradiction.

Let $P$ be the path of length $2k-1$ from $y$ to $z$ in $D-zy$.  We now claim that in $G_{\text{int}}(D) - zy$, the path $P$ is a shortest $(y,z)$-path.  Suppose there is a shorter path $P'$.  Using the fact that $D$ bounds the outerface of $G_{\text{int}}(D)$ and $zy$ is an edge of $D$, the cycle $P'+zy$ is $f$-separating in $G$ and of length less than $2k$, a contradiction.
Therefore, by Lemma~\ref{shortestpath}, $G_{\text{int}} (D)- zy$ retracts to $P$ which implies $G_{\text{int}}(D)$ folds to $D$. (The vertices $y$ and $z$ are fixed under the retraction.)  Let $G''$ be the resulting graph from $G$ after folding $G_{\text{int}}(D)$ to $D$. In $G''$, $D$ is the boundary of a $\geq 2k$-face and the outerface is a $\geq 2k$-face.  (This includes the possibility that $G''=D$ is simply a cycle.) Now by induction the result follows.
\end{proof}

Together Lemma~\ref{lem:atmostonelongface} and Lemma~\ref{lem:atleasttwolongfaces} establish Theorem~\ref{thm:bipartiteplanarmixing}.  Since the property of being $(p,q)$-mixing is independent of any particular planar embedding we use, the following corollary is immediate.  

\begin{cor}\label{cor:allembeddingswork}
Fix $2 < \frac{p}{q} < 4$ and let $C_{2k}$ be the minimal non-$(p,q)$-mixing cycle.  Let $G$ be a $2$-connected, planar bipartite graph.  Then for all planar embedding of $G$ with no separating $C_{2i}$-cycle, $i \in \{ 2, \dots, k-1 \}$, either all such embeddings have at most one $\geq 2k$-face, in which case $G$ is $(p,q)$-mixing, or all such embeddings have at least two $\geq 2k$-faces, in which case $G$ is not $(p,q)$-mixing.
\end{cor}

Now we show we can perform reductions to remove small separating cycles. 
The first result handles all short separating cycles.

\begin{lem}
\label{easysplit}
Fix $2 < \frac{p}{q} <4$. Let $C_{2k}$ be the smallest even cycle which is not $(p,q)$-mixing. Let $G$ be a $2$-connected planar bipartite graph.  Suppose $G$ has a planar embedding where $C$ is a separating $C_{2i}$-cycle, for some $i \in \{ 2, \dots, k-1 \}$. If both $G_{1} = G_{\text{int}}(C)$ and $G_{2} = G_{\text{ext}}(C)$ are $(p,q)$-mixing, then $G$ is $(p,q)$-mixing. 
\end{lem}

\begin{proof}
Suppose that $G$ is not $(p,q)$-mixing. Let $\phi$ be a $G_{p,q}$-colouring where there is a wrapped cycle $D$. If $D$ contains vertices from both $G_{1}$ and $G_{2}$, then $D$ crosses the separating cycle $C$ (note, $D \neq C$ since $C$ is $(p,q)$-mixing). Then we can (repeatedly, if required) apply Lemma~\ref{chordlesscycle} to obtain a wrapped cycle which lies completely in $G_{1}$ or $G_{2}$, and thus either $G_{1}$ or $G_{2}$ is not $(p,q)$-mixing.
\end{proof}

We give an example below to show the converse of Lemma~\ref{easysplit} does not hold for all $2 < \frac{p}{q} < 4$.  Hence, we require a stronger assumption that $G$ has a separating $4$-cycle.

\begin{lem}
\label{nosep4cycle}
Let $G$ be a $2$-connected planar bipartite graph, and suppose it has a planar embedding with a separating four cycle $C$. If either $G_{\text{int}}(C)$ or $G_{\text{ext}}(C)$ is not $(p,q)$-mixing, then $G$ is not $(p,q)$-mixing.
\end{lem}

\begin{proof}
Without loss of generality assume $G_{\text{int}}(C)$ is not $(p,q)$-mixing.
By Corollary \ref{shortestcycleretract}, $G_{\text{ext}}(C)$ retracts to $C$. As every retract of a connected graph is a folding, this implies $G$ folds to $G_{\text{int}}(C)$, which is not $(p,q)$-mixing, and hence $G$ is not $(p,q)$-mixing.
\end{proof}

We need one more observation before we can finish the polynomial time algorithm.  

\begin{obs}
The graph $C_{6}$ is not $(p,q)$-mixing for any $3 \leq \frac{p}{q} <4$.
\end{obs}

\begin{proof}
Let $v_{0},v_{1},v_{2},v_{3},v_{4},v_{5}$ be the vertices of $C_{6}$, where $v_{i}v_{i+1} \in E(C_{6})$ for $i \in \{0,\ldots,5\}$, indices taken modulo $6$. Let $f$ be the $(p,q)$-colouring where $f(v_{0}) = f(v_{3}) = 0$, $f(v_{1}) = f(v_{4}) = q$, $f(v_{2}) = f(v_{5}) = 2q$. This is a proper $(p,q)$-colouring as $\frac{p}{q} \geq 3$. Observe that $\frac{|E(C)|}{2}p = 3p$. Orienting $C$ from $v_{i}$ to $v_{i+1}$ for $i \in \{0,\ldots,5\}$, we have $W(C,f) = 2p < 3p$, and hence $C_{6}$ is not $(p,q)$-mixing. 
\end{proof}

Now it follows from the above sequence of claims that there is a polynomial time algorithm for determining if a bipartite planar graph is $(p,q)$-mixing when $3 \leq \frac{p}{q} <4$, proving Theorem~\ref{introplanarity}.

To see this, one first finds a planar embedding. Secondly, find all two connected blocks of the graph. Thirdly, enumerate all separating $4$-cycles and reduce to the case where there are no separating $4$-cycles. Lastly, check if the resulting graphs are $(p,q)$-mixing by testing the size of the faces in the embedding. As $p$ and $q$ are fixed, each of these steps can be done in polynomial time (with the running time depending on $p$ and $q$). To obtain a precise running time, one can do analysis similar to that in~\cite{Cerecedathesis}.

We finish by showing the need to restrict our reduction to separating four cycles in the last result is necessary. For $(5,2)$-mixing the shortest even cycle that is not mixing is $C_{10}$.  We now give an example of a $2$-connected planar bipartite graph $G$ with a planar embedding containing a separating $8$-cycle $C$ such that $G_{\text{int}}(C)$ is not $(5,2)$-mixing but the entire graph $G$ is $(5,2)$-mixing. We will use the following observation:

\begin{obs}
\label{foldinginneighbourhoods}
Fix positive integers $p,q$ such that $2 < \frac{p}{q} < 4$. Let $G$ be a connected graph, and $u$ be a vertex whose neighbourhood is contained in the neighbourhood of a vertex $v$. Let $G_{uv}$ be the graph obtained from folding $u$ and $v$. Then $G$ is $(p,q)$-mixing if and only if $G_{uv}$ is $(p,q)$-mixing. 
\end{obs}

\begin{proof}
If $G$ is $(p,q)$-mixing, then Lemma \ref{foldsto} implies that $G_{uv}$ is $(p,q)$-mixing.

Now suppose that $G$ is not $(p,q)$-mixing. Then there is a $(p,q)$-colouring $f$ of $G$ and a cycle $C$ in $G$ which is wrapped with respect to $f$. As the neighbourhood of $u$ is contained in the neighbourhood of $v$, we may recolour $u$ so that $f(u) = f(v)$.  Since $2 < \frac{p}{q} < 4$, such a recolouring does not change the weight of any cycle, and so $C$ is still wrapped~\cite{circularReconfig}.
We define the $(p,q)$-colouring $f'$ of $G_{uv}$ by $f'(x) = f(x)$ for all $x \in V(G) \setminus \{u,v\}$ and if $x_{uv}$ is the vertex obtained from folding $u$ and $v$, we let $f'(x_{uv}) = f(v)$. That is,
$f = f' \circ g$ where $g$ is the fold that takes $u$ to $v$.  Thus, $W(C,f) = W(C,f' \circ g)$.  Since $g(C)$ is either a cycle in $G_{uv}$ or two cycles sharing the vertex $x_{uv}$,  by a similar argument as given in Lemma \ref{foldsto}, there is a wrapped cycle in $G_{uv}$, showing that $G_{uv}$ is not $(p,q)$-mixing. 
\end{proof}

\begin{figure}
\begin{tikzpicture}

\foreach \i in {0,2,4,6}
  \node[whitevertex] (v\i) at (90-\i*45:1.5cm) {};
\foreach \i in {1,3,5,7}
  \node[blackvertex] (v\i) at (90-\i*45:1.5cm) {};  
 
\node[label={0:$u$}] at (v0) {};
\node[label={0:$v$}] at (v4) {};
 
\draw[thick,black] (v0)--(v1)--(v2)--(v3)--(v4)--(v5)--(v6)--(v7)--(v0);  

\node[grayvertex,label={180:$w$}] (w) at (-3,0) {};
\draw[thick,black] (w)--(v6){};
\draw[thick,black] (w) to[out=45,in=135] (v0);
\draw[thick,black] (w) to[out=315,in=225] (v4);
\draw[thick,black] (w) to[out=270,in=180] (-0.5,-2.5) to[out=0,in=315] (v2);

\foreach \i in {1,3,5}
  \node[grayvertex] (u\i) at (0,1.5-\i*0.5) {};
\foreach \i in {2,4}
  \node[grayvertex] (u\i) at (0,1.5-\i*0.5) {};
\draw[thick,black] (v0)--(u1)--(u2)--(u3)--(u4)--(u5)--(v4);  
  
\end{tikzpicture}
\caption{An example of a planar graph $G$ with a separating eight cycle $C$ consisting of the
white vertices $A$ and the black vertices $B$.  The subgraph $G_{\text{int}}(C)$ is not 
$(5,2)$-mixing, but $G$ is $(5,2)$-mixing.}\label{fig:example}
\end{figure}
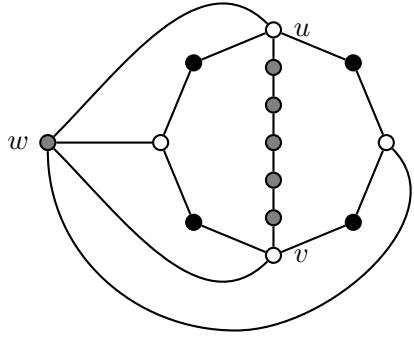

We construct our example $G$ as follows. Let $C$ be an $8$-cycle with bipartition $(A,B)$. Add a new vertex $w$ adjacent to all vertices in $A$. Let $u,v$ be two vertices in $A$ where $d(u,v) = 4$ in $C$. Add a path $P$ from $u$ to $v$ of new vertices such that $P$ has length $6$. Observe that $G$ is planar. Consider the embedding where $w$ lies on the exterior of $C$ and $P$ lies on the interior of $C$. Then $G_{\text{int}}(C)$ is not $(5,2)$-mixing, since it has no separating $i$-cycle for $i \in \{4,6,8\}$ and two faces of length $10$. Notice that each vertex of $B$ has its neighbourhood in the neighbourhood of $w$.  Hence, we can fold each vertex of $B$ to $w$.  Next we can fold each vertex of $A \setminus \{ u, v \}$ to $u$.  By repeatedly applying Observation~\ref{foldinginneighbourhoods}, we see $G$ is $(5,2)$-mixing if and only if $C_{8}$ is $(5,2)$-mixing, which is true.  See Figure~\ref{fig:example}.

For $p/q < 3$, $C_6$ is $(p,q)$-mixing, so the shortest non-mixing cycle has length at least $8$.  Using this, the example easily generalizes to any maximal mixing cycle $C_{2k}$ by joining $u$ and $v$ with a path of length $k+2$ and adding $w$ dominating one part (of the $C_{2k}$ bipartition) so that the resulting graph folds to a cycle that is $(p,q)$-mixing.

\section{Conclusion}

The study of \Hrec{K_3} and \Hmix{K_3} in~\cite{3colReconfig,Bonsma,Mixing3Col,Cerecedathesis} identifies several key ideas which extend to the study of \Hrec{G_{p,q}} and \Hmix{G_{p,q}} for $2 < \frac{p}{q} < 4$.  While some of the results depend on $\frac{p}{q} < 4$, others depend on the homomorphism target being $2$-regular or containing a triangle.  These notions all coincide for $K_3$, but differ for $2 < \frac{p}{q} < 4$.  This paper represents an initial study into \Hmix{H} in this finer landscape.  It will be interesting to develop further methods to provide a complete classification of the complexity of \Hmix{H} for $2 < \frac{p}{q} < 4$.

\bibliography{Reconfigv2}
  \bibliographystyle{plain}

\end{document}